\documentclass[12pt,reqno]{amsart}

\usepackage{amssymb}
\usepackage{longtable}

\theoremstyle{plain}
\newtheorem{theorem}{Theorem}
\newtheorem{lemma}{Lemma}

\theoremstyle{proof}
\theoremstyle{definition}

\newtheorem{definition}{Definition}[section]

\theoremstyle{remark}

\theoremstyle{lamma}

\numberwithin{equation}{section}
\numberwithin{lemma}{section}
\numberwithin{theorem}{section}

\usepackage{amssymb, amsmath, amsthm}
\usepackage[breaklinks]{hyperref}
\theoremstyle{thmrm}

\usepackage{graphicx}

\begin{document} 
\title[Certain Diophantine tuples in imaginary quadratic fields]{Certain Diophantine tuples in imaginary quadratic fields}
\author{Shubham Gupta}
\address{Shubham Gupta @Harish-Chandra Research Institute, HBNI,
Chhatnag Road, Jhunsi, Allahabad 211 019, India.}
\email{shubhamgupta@hri.res.in}
\keywords{Diophantine tuples, Imaginary quadratic fields, Pell equation, Simultaneous approximation.}
\subjclass[2010] {Primary: 11D09, 11R11, Secondary: 11J68.}
\maketitle

\begin{abstract} Let $K$ be an imaginary quadratic field and $
\mathcal{O}_K$ be its ring of integers. A set $\{a_1, a_2,
\cdots,a_m\} \subset \mathcal{O}_K\setminus\{0\}$ is called a 
Diophantine $m$-tuple in $\mathcal{O}_K$ with $D(-1)$ if 
$a_ia_j -1 = x_{ij}^2$, where $x_{ij} \in \mathcal{O}_K$ for all 
$i,j$ such that $1 \leq i < j \leq m$. Here we prove 
the non-existence of Diophantine $m$-tuples in $\mathcal{O}_K$ with 
$D(-1)$ for $m > 36$.
\end{abstract}

\section{Introduction} 
A set $\{a_1,a_2, \cdots,a_m\}$ of $m$ positive integers is 
called a Diophantine $m$-tuple with $D(n)$ if $a_ia_j + n 
= x_{ij}^2$, where $x_{ij} \in \mathbb{Z}$ and $n \in \mathbb{Z}
$, for all $1 \leq i < j \leq m$. Diophantus found a set of 
four positive rationals \{1/16, 33/16, 17/4, 105/16\} with the 
above property for $n = 1$. The first Diophantine $4$-tuple with $D(1)$, namely, $\{1, 3, 8, 120\}$
was found by Fermat.
Baker and Davenport \cite{BD1969} proved that this 
particular quadruple cannot be extended to a Diophantine $5$-tuple 
with $D(1)$. Now on whenever we say a $m$-tuple, it would mean a Diophantine $m$-tuple as above.

Let $\{a, b, c\}$ be a $3$-tuple with $D(1)$. If there exists a $d \in \mathbb{N}$ such that $\{a, b, c, d\}$ is a $4$-tuple with $D(1)$, then there exist $x,y,z$ $\in$ $\mathbb{Z}$ 
such that 
$$
ad + 1 = x^2,\hspace{0.2cm} bd + 1=y^2, \hspace{0.2cm} \text{and } cd + 1 
=z^2.
$$
Hence we get an elliptic curve $E$ over $\mathbb{Q}$
$$
E : (xyz)^2=(ad+1)(bd+1)(cd+1).
$$
As the number of integral points on an elliptic curve over $
\mathbb{Q}$ is finite(\cite[page~176]{ST1992}) so the number of possible choices of $d$ is finite. Over the years due to the findings of many researchers there exist many examples of $3$- and $4$-tuples. In 2001, Dujella \cite{DU2001} proved that there are atmost finitely many Diophantine 8-tuple with $D(1)$ and there does not exist Diophantine $9$-tuple with $D(1)$. In 2004, he improved this result and proved that there does not exist Diophantine 6-tuple with $D(1)$ and there exist atmost finitely many Diophantine 5-tuple with $D(1)$ (see \cite{DU2004}).
There was a `folklore' conjecture that there 
does not exist Diophantine $5$-tuples with $D(1)$. This is recently (in 2019) been settled by B. He et. al.
\cite{HTZ2019} in a pioneering work. Let
$$
S(n) = \max \{|A|: A \text{ is a Diophantine} ~~m-\text{tuple with} ~D(n) \}.
$$
Thus from the work of He et.al. $S(1) \leq 4$. 
 Dujella and Fuchs \cite{DF2005} showed that there do not exist Diophnatine 5-tuples with $D(-1)$. Dujella, Fuchs and Filipin \cite{DFF2007} also proved that there exist atmost finitely many Diophnatine 4-tuple with $D(-1)$. Furthermore they showed that, any such Diophantine 4-tuple with $D(-1)$ $\{a_1,\cdots,a_4\}$ should satisfy $a_4<10^{903}$. This bound was further reduced to $3.01 \times 10^{60}$ by Trudgian \cite{TR2015}.
 
\begin{definition}\label{Def_1}
A set $\{a_1, a_2,\cdots, a_m\} \subset \mathcal{O}_K \setminus \{0\}$ is called Diophantine $m$-tuples in $\mathcal{O}_K$ with $D(n)$ 
if $a_ia_j + n = x_{ij}^2$, $x_{ij} \in \mathcal{O}_K$ for all $1 \leq i < j \leq m$.
\end{definition}

For the remainder of the article, $m$ and $n$ carry the same meaning as in definition \ref{Def_1} above.

In 1997, Dujella proved that there does not exist Diophantine 4-tuple in $\mathbb{Z}[i]$ with $D(a+bi)$ , where $b$ is odd or $a \equiv b \equiv 2 \pmod{4}$ (see \cite{DU1997}). For $n = 1$, Azadaga \cite{AD2019} proved that $m \leq 42$. For $n=-1$, Soldo studied the extension of certain triples to quadruples (see \cite{SO2013}, \cite{SO2016}). In this paper, we studied the existence of $m$-tuple with $D(-1)$ and obtained the following:

\begin{theorem}\label{thm1}
Let $K$ be an imaginary quadratic field and $\mathcal{O}_K$ be its ring of integers. Then there does not exist Diophantine $m$-tuple with $D(-1)$ for $m > 36$ in $\mathcal{O}_K$. 
\end{theorem} 

Here is a brief of how we proceed to prove the above result. We employ similar techniques as that of Azadaga \cite{AD2019}. Let $\{a, b, c\}$ be a triple in $\mathcal{O}_K$ with $D(-1)$. If $d \in \mathcal{O}_K$ such that $\{a, b, c, d\}$ be a quadruple with $D(-1)$,
then we get a system of Pellian equations. Using the solution of these Pellian equations and a result of Jedrizevi\'{c}-Zeigler \cite{JZ2006}, we will get an upper bound on $d$ in term of $c$, if $\{a, b, c, d\}$ satisfies some conditions. 
Further using the regularity condition (refer section \ref{Sec_4} below) on $\{a, b, c, d\}$ one gets a lower bound, i.e., $ d \geq g(a)$ for some function $g$ in terms of $a$. We use SAGE for the computations and prove Theorem \ref{thm1} by contradiction. The lower and upper bounds on $d$ will give the desired contradiction.

\section{System of Pellian equations}
Let $K = \mathbb{Q}(\sqrt{-D})$ with $D$ a square free positive integer. We know that $\mathcal{O}_K = \mathbb{Z}[\omega] = \{a + b\omega: a, b \in \mathbb{Z}\}$, where 
$$
\omega=\begin{cases}
  \sqrt{-D} & \text{if $-D \equiv 2,3 \pmod{4}$},\\
  \dfrac{1+\sqrt{-D}}{2} & \text{if $-D \equiv 1 \pmod{4}$}.
 \end{cases}
$$
If $\alpha = \left(a + \dfrac{b}{2}\right) + \dfrac{ b }{2} \sqrt{-D} \in \mathcal{O}_K$ then the norm of $\alpha$:
$$
||\alpha|| = \left(a+\dfrac{b}{2}\right)^2+\dfrac{Db^2}{4},
$$
and in particular if $ \alpha = a + b\sqrt{-D}$, 
then 
$$
||\alpha|| = a^2 + Db^2.
$$
Then the absolute value of $\alpha \in \mathcal{O}_K$ (denoted as $|\alpha|$) is defined as 
$|\alpha| = \sqrt{||\alpha||}$.
When $D = 1$ the units in $\mathbb{Z}[i]$ are $\{\pm 1, \pm i\}$, when $D=3$ the units are $\left\{\pm 1, \dfrac{\pm 1 \pm \sqrt{-3}}{2}\right\}$ and else the units are $\{\pm 1\}$. 

\textbf{Notations-} Throughout, a triple $\{a, b, c\}$ will denote a Diophantine $3$-tuple in $\mathcal{O}_K$ such that $0<|a|\leq |b| \leq |c|$ with property $D(-1)$ and similarly other tuples.
Let $r, s, t \in \mathcal{O}_K$ such that
$$
r=\sqrt{ab-1},\ s = \sqrt{ac-1} ~~\text{and}~~ t = \sqrt{bc-1},
$$
where $a, b, c, d$ form a quadruple.
\begin{lemma}
Let $\mathcal{A} = \{a_1, a_2, a_3, \cdots, a_m\}$ be a $m$-tuple in $\mathcal{O}_K$ with $D(-1)$. Then, for $m \geq 4$, $a_ia_j$ is not a square in $\mathcal{O}_K$ for all $1 \leq i < j \leq m$. Also, for $m \geq 4$, $a_ia_j$ is not a square in $K$.
\end{lemma}

\begin{proof}
If $\{a, b\}$ be a pair in $\mathcal{A}$ such that $ab = x^2$ where $x \in \mathcal{O}_K \setminus \{0\}$, then
$$
ab-1=r^2=x^2-1 \Rightarrow 1=x^2-r^2=(x-r)(x+r) \Rightarrow x=0 \hspace{0.2cm} \text{or} \hspace{0.2cm} r=0,
$$
so $r=0$ and hence $ab = 1$. If $D = 1$ then 
 $a, b \in \{i,-i\}$ and it implies that if $\{a, b, c\}$ be a triple then $c$ has to be one of $\{ \pm i\}$. One can easily check that $\{a, b, c\}$ is not triple in $\mathcal{O}_K$ with $D(-1)$. On the other hand when $D=3$ then 
$a, b \in \Big\{\dfrac{ \pm 1 \pm \sqrt{-3}}{2}\Big\}$. It implies that if $\{a, b, c\}$ be a triple then, $c$ is one of $\{\pm 1\}$. Thus only two pairs
$\Big\{\dfrac{1 + \sqrt{-3}}{2}, \dfrac{1 - \sqrt{-3}}{2}\Big\}$ and $\Big\{\dfrac{-1 + \sqrt{-3}}{2}, \dfrac{-1 - \sqrt{-3}}{2}\Big\}$ survive. The corresponding triples are 
$$
\Big\{\dfrac{1 + \sqrt{-3}}{2}, \dfrac{1 - \sqrt{-3}}{2}, 1\Big\}~~ \text{and}~~ \Big\{\dfrac{-1 + \sqrt{-3}}{2}, \dfrac{-1 - \sqrt{-3}}{2}, -1\Big\}.
$$
Note also that these pairs $\Big\{\dfrac{1 + \sqrt{-3}}{2}, \dfrac{1 - \sqrt{-3}}{2}\Big\}$ and $\Big\{\dfrac{-1 + \sqrt{-3}}{2}, \dfrac{-1 - \sqrt{-3}}{2}\Big\}$ cannot be extended to quadruple. Now if $D \neq 1,3$ then the units are $\pm 1$ so either $a=b=1$ or $a=b=-1$. Hence $ab$ is not a square in $\mathcal{O}_K$.

Now if $ab$ is a square in $K$, then it is a root of monic polynomial $x^2-ab$. Since $\mathcal{O}_K$ is integrally closed, $ab$ is not a square in $K$. Hence $ab$ is not a square in $K$.
\end{proof}
Let us suppose $\{a, b, c\}$ extends to a quadruple $\{a, b, c, d\}$. Thus there exist $x,y,z \in \mathcal{O}_K$ such that \[ad-1=x^2,\hspace{0.2cm} bd-1=y^2,\hspace{0.2cm} cd-1=z^2.\] 
Thus there is a system of Pell's equations:
\begin{equation}\label{eqq8}
az^2-cx^2=c-a
\end{equation}
\begin{equation}\label{eqq9}
bz^2-cy^2=c-b
\end{equation}
with \( d = \dfrac{z^2+1}{c}\).
\section{Upper bound of $d$ in term of $c$}
Let $\{a, b, c, d\}$ be a quadruple.
We will see that if $c$ is bounded by some power of $b$ then $d$ is bounded by some power of $c$. In 1998, Bennett \cite{BE1998} proved a theorem which is related to simultaneous approximations of rationals, where these rationals have square roots close to one.
Jadrijevi\`c-Zeigler proved the following theorem which is an analog to Bennett's theorem.

\begin{lemma}$($Jadrijevi\'c -Zeigler \cite[Theorem~7.3, 7.4]{JZ2006}$)$ \label{lem1}
Let $\theta_i = \sqrt{1+\dfrac{a_i}{T}}$, $i=1,2$ with $a_1, a_2$ distinct algebraic integers in $K$, and $T$ be any algebraic integer of $K$. Further, let $M=\max\{|a_1|,|a_2|$\}, $|T|>M$, $a_0 = 0$ and 
$$
L=\dfrac{27}{16|a_1|^2|a_2|^2|a_1-a_2|^2}(|T|-M)^2>1
$$
Then
\begin{equation}
\max \{|\theta_1-p_1/q|, |\theta_2-p_2/q|\} > c_1|q|^{-\lambda} 
\end{equation}
for all algebraic integers $p_1,p_2,q \in K$ 
where
\begin{align*}
\lambda & =1+\dfrac{\log P}{\log L},\hspace{0.2cm}c_1^{-1} =4pP(\max \{1,2l\})^{\lambda - 1},\\
l & =\dfrac{27|T|}{64(|T|-M)}, \hspace{0.2cm} p = \sqrt{\dfrac{2|T|+3M}{2|T|-2M}},\\
P & =16\dfrac{|a_1|^2|a_2|^2|a_1-a_2|^2}{\min\{|a_1|, |a_2|, |a_1-a_2|\}^3}(2|T|+3M).
\end{align*}

\end{lemma}

\begin{lemma}\label{lem2}
Let $(x,y,z)$ be a solution of the system of equations \eqref{eqq8} and \eqref{eqq9}. Assume $|c|>4|b|$, $|a|\geq 2$. If 
$\theta_1^{(1)} = \pm \dfrac{s}{a}\sqrt{\dfrac{a}{c}}, \hspace{0.2cm} \theta_1^{(2)} = - \theta_1^{(1)}$ and $\theta_2^{(1)} = \pm \dfrac{t}{b}\sqrt{\dfrac{b}{c}}, \hspace{0.2cm}\theta_2^{(2)} = - \theta_2^{(1)}$ with `sign' chosen so that
\(\Big|\theta_1^{(1)} - \dfrac{sx}{az}\Big| \leq \Big|\theta_1^{(2)} - \dfrac{sx}{az} \Big|\) and \(\Big|\theta_2^{(1)} - \dfrac{ty}{bz} \Big| \leq \Big|\theta_2^{(2)} - \dfrac{ty}{bz} \Big|\), then 
\begin{equation}\label{eqq6}
\Big|\theta_1^{(1)} - \dfrac{sbx}{abz}\Big| \leq 
\dfrac{|s||a-c|}{|a|\sqrt{|ac|}} \times \dfrac{1}{|z|^2} < \dfrac{21|c|}{16|a|} \times \dfrac{1}{|z|^2}
\end{equation} 
and
\begin{equation}\label{eqq7}
 \Big|\theta_2^{(1)} - \dfrac{tay}{abz}\Big| \leq \dfrac{|s||a-c|}{|b|\sqrt{|bc|}} \times \dfrac{1}{|z|^2} < \dfrac{21|c|}{16|a|} \times \dfrac{1}{|z|^2}.
 \end{equation} 
\end{lemma}

\begin{proof}
We prove inequality \eqref{eqq6} and similarly \eqref{eqq7} can be proven. Consider
$$
\Big|\theta_1^{(1)} - \dfrac{sx}{az}\Big| = \dfrac{\Big|\theta_1^{(1)} - \dfrac{sx}{az}\Big| \times \Big|\theta_1^{(1)} + \dfrac{sx}{az}\Big|}{\Big| \theta_1^{(1)} + \dfrac{sx}{az}\Big|} = \dfrac{\Big|\Big(\theta_1^{(1)}\Big)^2 - \dfrac{s^2x^2}{a^2z^2}\Big|}{\Big|\theta_1^{(1)} + \dfrac{sx}{az} \Big|}.
$$
We substitute $\theta_1^{(2)} = - \theta_1^{(1)}$ in above and get
\begin{align*}
\dfrac{\Big|\Big(\theta_1^{(1)}\Big)^2 - \dfrac{s^2x^2}{a^2z^2}\Big|}{\Big|\theta_1^{(1)} + \dfrac{sx}{az} \Big|} 
&= \Big| \dfrac{s^2}{a^2 }\Big| \times \Big|\dfrac{a^2}{s^2} \times \Big(\theta_1^{(1)}\Big)^2 - \dfrac{x^2}{z^2}\Big| \times\Big|\theta_1^{(2)} - \dfrac{sx}{az} \Big|^{-1}\\ 
&=\Big|\dfrac{s^2}{a^2}\Big|\times \Big|\dfrac{a}{c} - \dfrac{x^2}{z^2}\Big| \times \Big|\theta_1^{(2)} - \dfrac{sx}{az} \Big|^{-1}\\
&= \Big|\dfrac{s^2}{a^2}\Big|\times \Big|\dfrac{az^2 - cx^2}{|cz^2|}\Big| \times \Big|\theta_1^{(2)} - \dfrac{sx}{az} \Big|^{-1}\\
&=\Big|\dfrac{s^2}{a^2}\Big|\times \dfrac{|c - a|}{|cz^2|} \times \Big| \theta_1^{(2)} - \dfrac{sx}{az} \Big|^{-1}.
\end{align*} 
This is because
\begin{align*}
2\Big|\theta_1^{(2)} - \dfrac{sx}{az}\Big| 
& \geq \Big| \theta_1^{(2)} - \dfrac{sx}{az}\Big| + \Big| \theta_1^{(1)} - \dfrac{sx}{az}\Big|\\
& \geq \Big| \theta_1^{(2)} - \dfrac{sx}{az} - \Big( \theta_1^{(1)} - \dfrac{sx}{az}\Big)\Big|\\
& = \Big| \theta_1^{(2)} - \theta_1^{(1)}\Big| = 2\Big| \dfrac{s}{a}\sqrt{\dfrac{a}{c}}\Big|.
\end{align*}
Thus 
$$
\Big|\theta_1^{(2)} - \dfrac{sx}{az}\Big| \geq \Big| \dfrac{s}{a}\sqrt{\dfrac{a}{c}}\Big|.
$$
This implies that \[\Big|\theta_1^{(1)} - \dfrac{sbx}{abz}\Big| \leq \dfrac{|s||c-a|}{|a|\sqrt{|ac|}}\times\dfrac{1}{|z|^2}.\]
For proving other part of the inequality \eqref{eqq6}, we want to show that
$$
|\sqrt{ac - 1}| \times |c - a| < (21/16) \times |c| \times \sqrt{|ac|} 
$$
and this holds if and only if 
$$
 \Big|\sqrt{1 - \dfrac{1}{ac}} \Big| < \dfrac{21}{16} \times\dfrac{|c|}{|c - a|}.
$$
Now $|c| > 4|a|$ implies that 
$$
\dfrac{21}{16} \times \dfrac{|c|}{|c - a|} \geq \dfrac{21}{20}
$$ 
and then 
\begin{eqnarray*}
\Big|\sqrt{1 - \dfrac{1}{ac}} \Big| &=& \sqrt{\Big|1 - \dfrac{1}{ac}\Big|}\\
&\leq& \sqrt{1 + \dfrac{1}{|ac|}} < \dfrac{\sqrt{17}}{4}\\
& <& \dfrac{21}{20} \\
&\leq& \dfrac{21}{16} \times\dfrac{|c|}{|c - a|}.
\end{eqnarray*}
\end{proof}

Thus from Lemma \ref{lem2} we conclude that
\begin{eqnarray*}
\Big|\theta_1^{(2)} + \dfrac{sbx}{abz}\Big| &=& \Big|\theta_1^{(1)} - \dfrac{sbx}{abz}\Big| \\
&\leq& 
\dfrac{|s||a-c|}{|a|\sqrt{|ac|}} \times \dfrac{1}{|z|^2}\\ 
&<& \dfrac{21|c|}{16|a|} \times \dfrac{1}{|z|^2}, 
\end{eqnarray*}
and
\begin{eqnarray*}
\Big|\theta_2^{(2)} + \dfrac{tay}{abz} \Big| &=& \Big|\theta_2^{(1)} - \dfrac{tay}{abz}\Big| \\
&\leq& \dfrac{|s||a-c|}{|b|\sqrt{|bc|}} \times \dfrac{1}{|z|^2}\\
&<& \dfrac{21|c|}{16|a|} \times \dfrac{1}{|z|^2}.
\end{eqnarray*}

\begin{lemma}\label{lem5}
Let \{a, b, c, d\} be a quadruple such that $|b| \geq (3/2)|a|$, $|b| \geq22$, $|a| \geq 2$ and $|c| > |b|^{16}$. Then 
$$
|d| < (3956)^{10}|c|^{24}.
$$
\end{lemma}

\begin{proof}
Let $\theta_1=\dfrac{s}{a}\sqrt{\dfrac{a}{c}}$ and $\theta_2=\dfrac{t}{b}\sqrt{\dfrac{b}{c}}$. Then 
\begin{eqnarray*}
\theta_1 &=& \sqrt{\dfrac{s^2a}{a^2c}} = \sqrt{1 + \dfrac{(-b)}{abc}},~~\text{and}\\ 
\theta_2 &= &\sqrt{\dfrac{t^2b}{b^2c}} = \sqrt{1 + \dfrac{(-a)}{abc}}.
\end{eqnarray*}
If we write $a_1 = -b$, $a_2 = -a,$ $T = abc$ and $M = |b|$ then the claim is that:
$$
l = \dfrac{27|abc|}{64(|abc| - |b|)} < \dfrac{1}{2}.
$$
Proving the above claim is equivalent to show that $27|abc| < 32 (|abc| - |b|)$ and this holds if and only if $|ac| > (32/5)$. By hypothesis $|ac| \geq |b| \geq 22 > (32/5)$ and thus the claim holds.

Now
$$
p = \sqrt{\dfrac{2|abc| + 3|b|}{2|abc| - 2|b|}} = \sqrt{1 + \dfrac{5}{2(|ac| - 1)}} \leq \sqrt{\dfrac{47}{42}}.
$$
Also
$l < \dfrac{1}{2}$, one has $c_1^{-1} = 4pP \times 1$
would give
$$
c_1 \geq \dfrac{1}{4 \times P \times (\sqrt{47/42})} = \dfrac{\sqrt{42}}{\sqrt{47}(4 P)}.
$$
Consider now
$$
P = 16 \times \dfrac{|-b|^2|-a|^2|-b + a|^2}{\min \{|-a|, |-b|, |-a + b|\}^3}\times \Big(2|abc| + 3|b|\Big).
$$
Since 
$$
|-b + a| \geq |b| - |a| \geq \Big(\dfrac{3}{2} \times |a| - |a|\Big) = \dfrac{|a|}{2},
$$
so, $\min \{|a|, |b|, |a - b|\} \geq \dfrac{|a|}{2}$. Thus 
$$
P \leq 128 \cdot \dfrac{|b|^2|a|^2|b - a|^2|b|(2|ac| + 3)}{|a|^3}.
$$
Hence
\begin{equation}\label{eq6}
P \leq \dfrac{128|b|^3|b - a|^2(2|ac| + 3)}{|a|}.
\end{equation}
Let us now look at
$$
L = \dfrac{27}{16|-b|^2|-a|^2|-b + a|^2} \times \Big(|abc| - |b|\Big)^2 = \dfrac{27(|ac| - 1)^2}{16|a|^2|b - a|^2}.
$$
We claim that $L >1$. Which is 
equivalent to show 
$27(|ac| - 1)^2 > 16 |a|^2|b - a|^2$. This holds if and only if $3\sqrt{3}(|ac - 1|) > 4|a||b - a|$ which is equivalent to 
$$
\dfrac{3\sqrt{3}}{4} \times (|ac| - 1) > |a||b - a|.
$$
Since 
$$
|ac| - 1 > |a||b|^3 - 1 > 2|a|^2|b| - 1 > |a||b| + |a|^2 \geq |ab - b^2| = |a||b - a|
$$
the claim is validated.

Clearly $P > 1$ and so $\lambda > 1$. In fact
$ \lambda < 1.8$.

 Indeed, observe that $\lambda = 1 + \dfrac{\log P}{\log L} < 1.8$ holds if and only if $P < L^{0.8}$ , which is equivalent to 
 $$ 
 P < \Big(\dfrac{27}{16}\Big)^{0.8}\times \Bigg( \dfrac{|ac| - 1}{|a|(|b - a|)}\Bigg)^{1.6}.
 $$
 Appealing to inequality \eqref{eq6}, 
we need to show 
 $$ 
 \dfrac{128|b|^3|b - a|^2(2|ac| + 3)}{|a|} < \Big(\dfrac{27}{16}\Big)^{0.8} \cdot \Big(\dfrac{|ac| - 1}{|a||b - a|} \Big)^{1.6}.
 $$
After rearranging the above inequality,
 $$
128|b|^3|b - a|^{3.6}|a|^{0.6}(2|ac| + 3) < \Big(\dfrac{27}{16}\Big)^{0.8}(|ac| - 1)^{1.6}.
 $$
We see that it suffices to show
 \begin{equation}\label{eq7}
128|b|^3|b - a|^{3.6}(9/4)|a|^{0.6} < \Big(\dfrac{27}{16}\Big)^{0.8}(|ac| - 1)^{0.6},
 \end{equation}
as $|ac| - 1 > \dfrac{4}{9}(2|ac| + 3)$. 
Since the function $f(t) = (t - 1)^{0.6} - t^{0.6} + 1$
vanishes at $t = 1$ and is increasing, $|ac|^{0.6} - 1 < (|ac| - 1)^{0.6}$. 
Thus (using $|c| > |b|^{16}$) 
$$
 |a|^{0.6}|b|^{9.6} - 1 = |a|^{0.6}|b|^{(16 )\cdot (0.6)} - 1 < |ac|^{0.6} - 1 < (|ac| - 1)^{0.6}.
$$
For proving inequality \eqref{eq7}, it suffices to show 
\begin{equation}\label{eqq1}
128 \times (9/4)|b|^3|b - a|^{3.6}|a|^{0.6} < \Big(\dfrac{27}{16}\Big)^{0.8} (|b|^{9.6} - 1).
\end{equation}

Since we have $|a| \leq \dfrac{2}{3}(|b|)$, 
\begin{eqnarray*}
\Big(\dfrac{16}{27}\Big)^{0.8} \times 128 \times (9/4)|b|^3|b - a|^{3.6}|a|^{0.6} &<& \Big(\dfrac{16}{27}\Big)^{0.8} \times 128 \times (9/4)|b|^3(5/3)^{3.6}\cdot|b|^{3.6}\cdot \Big|\dfrac{2b}{3} \Big|^{0.6}\\ &<& 936|b|^{7.2}. 
\end{eqnarray*}
Thus inequality \eqref{eqq1}
holds if \(936|b|^{7.2} < |b|^{9.6} - 1\). This is obvious since  the function
$f(t) = t^{9.6} - 936t^{7.2} - 1$ is increasing function for $t \geq 15.5$ and $f(18)>0$. Hence our claim is proved.

Proceeding further, with $\theta_1, \theta_2$ as above, take $p_1 = \pm sbx , p_2 = \pm tay, q = abz$ (`sign' is chosen suitably) and upon applying Lemmas \ref{lem1} and \ref{lem2},
we get 
$$
\dfrac{21}{16}\cdot \dfrac{|c|}{|a|}\cdot\dfrac{1}{|z|^2} > \dfrac{\sqrt{42}}{\sqrt{47}(4P)}|abz|^{-\lambda}.
$$
From inequality \eqref{eq6}, we get 
$$
\dfrac{21}{16}\cdot \dfrac{|c|}{|a|}\cdot\dfrac{1}{|z|^2} > \dfrac{\sqrt{42}|a||abz|^{-\lambda}}{\sqrt{47} (4\cdot128) \cdot |b|^3|b - a|^2(2|ac| + 3)}. 
$$
It implies that 
$$
\dfrac{21}{16}\dfrac{4\sqrt{47} \times 128}{\sqrt{42}}\dfrac{|c|}{|a|^2}|b|^3|b - a|^2(2|ac| + 3) \cdot |ab|^{\lambda} > |z|^{2 - \lambda} > |z|^{0.2}.
$$
 Hence
$$ 
|z|^{0.2} < 712 |c|\cdot 3 \cdot|ac| |b - a|^2|b|^{3 + \lambda}|a|^{\lambda - 2} < 712 \times 3|c|^2 \cdot (2/3)|b|(5/3)^2|b|^2|b|^{4.8}.
 $$
 Using $ |c| < |b|^{16}$, one further gets,
$$
|z|^{0.2} < 3956 \cdot |c|^2|b|^{7.8} < 3956|c|^{2.49}.
$$ 
Hence 
$$
|z| < (3956)^{5}|c|^{12.45}
$$
and 
finally 
$$
|d| = \dfrac{|z^2 - 1|}{|c|} \leq \dfrac{|z|^2 + 1}{|c|} \leq \dfrac{(3956)^{10}|c|^{24.9} + 1}{|c|} < 3956^{10}|c|^{24}.
$$
\end{proof}

\section{Lower bound on $d$} \label{Sec_4}
A triple $\{a, b, c\}$ is said to be regular if $c = a + b \pm 2r$ (refer notation above). If $\{a,b,c,d\}$ is a quadruple, then the use of this regularity criterion gives us a lower bound on $d$ in terms of $a$. The following lemma states this.
\begin{lemma}\label{lem3}
Let $\{a, b, c, d\}$ be a quadruple with $5 < |a| \leq |b| \leq |c| \leq |d|$. Then atleast one of $\{a, b, c\}$ and $\{a, b, d\}$ is not regular. 
\end{lemma}

\begin{proof}
If possible let both $\{a, b, c\}$ and $\{a, b, d\}$ are regular, i.e., $c = a + b + 2r$ and $d = a + b - 2r$. Substituting the value of $r$ gives $cd - 1 = (a - b)^2 + 3$. As $\{c, d\}$ is a pair in $\mathcal{O}_K$ with $D(-1)$, there exists a $z \in \mathcal{O}_K$ such that $cd - 1 = z^2$. Thus $z^2 = (a - b)^2 + 3$ and therefore $3 = (z - (a - b))(z + (a - b))$. We take $X = (z -(a - b))$ and $Y = (z + (a - b))$. Then 
\begin{equation}\label{eqq3}
XY = 3
\end{equation}
and
\begin{equation}\label{eqq4}
X + Y = 2z.
\end{equation}
 Taking norm on both sides in \eqref{eqq3}, we get $||X||\times ||Y|| = ||3|| = 9$. 

\textit{Case (i)}: $||X|| = 1$ or $||Y|| = 1$.\\
Assume that $||X|| = 1$, then $X$ is a unit.\\
If $D = 1$, by equation \eqref{eqq3}, $(X, Y) \in \{ (1,3), (-1,-3), (i,-3i), (-i,3i)$\}. This implies that $X + Y = \pm 4, \pm 2i$ and therefore $z = \pm 2, \pm i$ (from the equation \eqref{eqq4}). Since $cd - 1 = z^2$, so either $cd = 5$ or $cd = 0$. Thus we get $|d| \leq 5,$ which is a contradiction to our hypothesis.\\
If $D = 3,$ by again using equation \eqref{eqq3}, we get 
\begin{equation*}
\begin{split}
(X, Y) \in \Bigg\{( 1,3),(-1,-3),\Big(\dfrac{1 + \sqrt{-3}}{2}, \dfrac{3(1 - \sqrt{-3})}{2}\Big),
\Big(\dfrac{1 - \sqrt{-3}}{2}, \dfrac{3(1 + \sqrt{-3})}{2} \Big), \\
 \Big(\dfrac{-1 + \sqrt{-3}}{2}, \dfrac{3(-1 - \sqrt{-3})}{2} \Big),\Big(\dfrac{-1 - \sqrt{-3}}{2}, \dfrac{3(-1 + \sqrt{-3})}{2} \Big)\Bigg\}.
 \end{split}
\end{equation*}
From equation \eqref{eqq4}, it follows that $2z = \pm 4,\pm 2 \pm \sqrt{-3}$. Since $z \in \mathcal{O}_K$, therefore $z = \pm 2$. Thus $cd = 5$. This implies that $|d| \leq 5$, a contradiction. \\
If $D \neq 1,3,$ then $(X, Y) \in \{(1,3),(-1,-3)\}$ (from equation \eqref{eqq3}). Again using equation \eqref{eqq4}, we get $2z = \pm 4$ and hence $cd = 5$. Again this will give $|d| \leq 5$, contradiction.\\
\textit{Case (ii)}: $||X|| = ||Y|| = 3$.\\
If $D = 1,$ then $||X|| = 3 = a_1^2 + b_1^2$ where $a_1, b_1 \in \mathbb{Z}$, which is not possible.\\
If $D = 2$, then $||X|| = 3 = a_1^2 + 2b_1^2$ where $a_1,b_1 \in \mathbb{Z}$. This implies that 
\begin{align*}
(X, Y) \in &\Big\{\Big(1 + \sqrt{-2}, 1 - \sqrt{-2} \Big), \Big(1 - \sqrt{-2}, 1 + \sqrt{-2} \Big), \\
&\Big(-1 + \sqrt{-2}, -1 - \sqrt{-2} \Big), \Big(-1 - \sqrt{-2}, -1 + \sqrt{-2} \Big)\Big\}.
\end{align*}

Then $z = \pm 1$ and therefore $cd = 2$. We conclude that $|d| \leq 2$.\\
If $D > 3$ and $D \equiv 1,2 \pmod{4},$ then
$||X|| = a_1^2 + Db_1^2 = 3$ where $a_1, b_1 \in \mathbb{Z}$ which is again not possible.\\
If $D = 3$, then $||X|| = \Big(a + \dfrac{b}{2}\Big)^2 + \dfrac{3 \cdot b^2}{4} = 3$. From equation \eqref{eqq3}, we get 
\begin{equation*}
\begin{split}
(X, Y) \in \Bigg\{\Big(\dfrac{3}{2} + \dfrac{\sqrt{-3}}{2}, \dfrac{3}{2} - \dfrac{\sqrt{-3}}{2} \Big), \Big(\dfrac{-3}{2} + \dfrac{\sqrt{-3}}{2}, \dfrac{-3}{2} - \dfrac{\sqrt{-3}}{2} \Big), \Big(\dfrac{3}{2} - \dfrac{\sqrt{-3}}{2}, \dfrac{3}{2} + \dfrac{\sqrt{-3}}{2} \Big), \\ \Big(\dfrac{-3}{2} - \dfrac{\sqrt{-3}}{2}, \dfrac{-3}{2} + \dfrac{\sqrt{-3}}{2} \Big), \Big(\sqrt{-3}, -\sqrt{-3}\Big), \Big(-\sqrt{-3}, \sqrt{-3}\Big)\Bigg\}.
\end{split}
\end{equation*}
Using equation \eqref{eqq4}, $2z = 0,\pm 3$. Since $z \in \mathcal{O}_K$, we get $z = 0$ and therefore $cd = 1$. This implies that $|d| \leq 1$, which is a contradiction. \\
Same way we can prove our lemma for $D \geq 7$ with $D \equiv 3 \pmod{4}$. 
\end{proof}

\begin{lemma}\label{lem4}
Let $\{a, b, c, d\}$ be a quadruple with $ 10 \leq |a| 
\leq |b| \leq |c| \leq |d|$, then $|d| \geq 
\dfrac{|ab|}{(330/65)} \geq \dfrac{|a|^2}{(330/65)}$.
\end{lemma}
\begin{proof}
We assume that 
$\{a, b, d\}$ is not regular(from Lemma \eqref{lem3}). Define 
$$
c_{\pm} = a + b + d - 2abd \pm 2rxy,
$$
where $x,y \in \mathcal{O}_K$ such that, 
$ad - 1 = x^2$ and $bd - 1 = y^2$.\\
\textit{Claim}: $c_{\pm} \neq 0$.\\
Suppose $c_{\pm} = 0$. This implies that $a + b + d(1 
- 2ab) = \mp 2rxy$. Squaring and rearranging this 
equation we get, $d^2 - 2d(a + b) + (a - b)^2 + 4 
= 0$. Therefore $d = a + b + 2r$ or $a + b - 2r$. 
Since $\{a, b, d\}$ is not regular, this is a contradiction.\\
Consider $c_+c_- = (a + b + d - 2abd)^2 - 4(rxy)^2 
= a^2 + b^2 + d^2 -2ab -2ad -2bd + 4$. Therefore $|c_
+c_-| \leq |d^2| + |d^2| + |d^2| + 2|d|^2 + 2|d|^2 
+ 2|d|^2 + |d|^2 \leq 10|d|^2$, also $|c_+ + c_-| = 2|a + b + d -2abd|$. We 
may assume that $|c_+| \geq |c_-|$. Since $2c_+ = 
|c_+| + |c_+| \geq |c_+ + c_-| = 2|a + b + d 
-2abd|$, this implies that, $$|c_+| \geq |a + b + d - 2abd|$$ 
We have $10 \leq |a| \leq |b| \leq |c| \leq |d|$, which follows that $|a + b + d| \leq 3|d| \leq 
\dfrac{3}{99}\cdot |abd|$. Thus 
$$
|c_+| \geq |a + b + d - 2abd| \geq 2|abd| - |a + b + d| \geq 2|abd| - (3/99)|abd| = \dfrac{65}{33}\cdot|abd|.$$ We have proved that $|c_+c_-| 
\leq 10|d|^2$ which gives that $|c_-| \leq 
\dfrac{10|d|^2}{|c_+|} \leq \dfrac{10|d|^2}{(65/33)|
abd|} = \dfrac{(330)|d|}{(65)|ab|}$. Since $c_- \neq 0$, $|c_-| \geq 1$ and this implies that $\dfrac{330|d|}
{65|ab|} \geq 1$. Hence $|d| \geq \dfrac{|ab|}{(330/65)} \geq 
\dfrac{|a|^2}{(330/65)}.$
\end{proof}

\section{Proof of the main theorem}

Let $\{a, b, c, d, e\}$ be a quintuple with $|e| < 
15$. For $D < 226$, we can check that, by 
computer, there does not exist such type of quintuples, and for $ D \geq 226$, we can easily 
seen that $a,b, c, d, e \in \mathbb{Z}$. Therefore, if $ab - 
1 = (x + y \sqrt{-D})^2$, then $2xy = 0$. This gives that 
either $x = 0$ or $y = 0$. Now if $x = 0$ then $ab - 1 
= -Dy^2$. This implies that $|ab - 1| \leq |ab| + 1 < 
226$, and hence $x = 0$ is not possible. Thus $y=0$. We conclude that if $\{a, b, c, d, e\}$ is a quintuple, then $|e| \geq 15$. Similarly, one can check that if $\{a, b, c, d\}$ is a quadruple, then $|d| \geq 12$.\\
Let $ \mathcal{A} = \{a_1, a_2, \cdots, a_m\}$ be a 
Diophantine $m$-tuple in $\mathcal{O}_K$ 
with $D(-1)$ such that $m \geq 37$. Thus $\{a_4, a_5, a_6, a_7\}$ is a quadruple. From Lemma \eqref{lem4}, we get $|a_7| \geq \dfrac{|a_4a_5|}{(330/65)} \geq \dfrac{12 \cdot 15}{(330/65)} > 35$. 

By applying lemma \eqref{lem4} to quadruples $\{a_7,a_8,a_9,a_{10}\}, \{a_{10},a_{11},a_{12},a_{13}\}$,$\cdots$, $\{a_{19}, a_{20},\\ a_{21}, a_{22}\}$ respectively, we get the following inequalities 
\begin{align*}
|
a_{10}| \geq \dfrac{|a_7|^2}{(330/65)},&& |a_{13}| \geq \dfrac{|a_{10}|^2}{(330/65)} = \dfrac{|
a_7|^4}{(330/65)^3},&& |a_{22}| \geq \dfrac{|a_7|^{32}}
{(330/65)^{31}}.
\end{align*}

Consider quadruples $\{a_4, a_7, a_{22}, a_{22 + 
k}\}$ for $k > 0$. Since $\{a_1, a_2, a_3, a_4\}$ is a quadruple, $|a_4| \geq 12$. Quadruple $\{a_4, a_5, a_6, 
a_7\}$ implies that $|a_7| 
\geq |a_5| \geq 15$ and from Lemma \eqref{lem4}, $|a_7| \geq \dfrac{|a_4a_5|}
{(330/65)} \geq \dfrac{15|a_4|}{(330/65)} > \dfrac{3|a_4|}{2}$. 

Inequality $|a_{22}| > |a_7|^{16}$ holds if $
\dfrac{|a_7|^{32}}{(330/65)^{31}} > |a_7|^{16}$, and 
this holds if $|a_7| > 24$. By Lemma\eqref{lem5}, 
\begin{equation}\label{eqq5}
|a_{22+k}| < 3956^{10}|a_{22}|^{24}, \hspace{0.2cm}\text{ $k >$ 0}.
\end{equation}

Again we apply lemma \eqref{lem4} to quadruples $\{a_{22}, a_{23}, a_{24}, a_{25}\}, \{a_{25}, a_{26}, a_{27}, a_{28}\}$, $\cdots$, $\{a_{34}, a_{35}, a_{36}, a_{37}\}$ respectively, and get the following inequalities 
\begin{align*}
|a_{25}| \geq \dfrac{|a_{22}|^2}{(330/65)},&& |a_{28}| \geq \dfrac{|a_{25}|^2}{(330/65)} 
\geq \dfrac{|a_{22}|^4}{(330/65)^3},&& |a_{37}| \geq \dfrac{|a_{22}|^{32}}
{(330/65)^{31}}.
\end{align*}
From inequality \eqref{eqq5},
\(3956^{10}|a_{22}|^{24} > |a_{37}| \). \\
\textit{Claim}: $\dfrac{|a_{22}|^{32}}{(330/65)^{31}} > 
3956^{10}|a_{22}|^{24}$.\\
It is equivalent to showing $|a_{22}|^{8} \geq 
(330/65)^{31} \cdot 3956^{10}$, and this inequality holds, if $|a_{22}| > 
1.8 \times 10^7$. Since $|a_{22}| \geq \dfrac{|
a_7|^{32}}{(330/65)^{31}} \geq \dfrac{35^{32}}{(330/65)^{31}} > 10^{27}$, our claim is proved.
Finally we get 
$$
3956^{10}|a_{22}|^{24} > |a_{37}| \geq 
\dfrac{|a_{22}|^{32}}{(330/65)^{31}} > 3956^{10}|a_{22}|
^{24},
$$
which is a contradiction. Hence $m \leq 36$. This completes the proof.

We have an example of quadruple in $\mathbb{Z}[i]$ with $D(-1)$ which is $\{1, 2, 5, -24\}$. Unfortunately, we do not know about the existence of  Diophantine $m$-tuple in $\mathcal{O}_K$ with $D(-1)$, for $m \geq 5$.

\section{Acknowledgement}

The author is indebted to Prof. Kalyan Chakraborty 
for his suggestions and for carefully going through 
the manuscript; The author is also thankful to Dr. A. 
Hoque for introducing him to this area and for his 
encouragement throughout. It is also a pleasure to 
acknowledge Mr. Mohit Mishra and Mr. Rishabh 
Agnihotri their support throughout the preparation of this manuscript and for providing all required assistance.

\end{document}